\documentclass{amsart}

\usepackage{amsmath,amssymb,amsthm,bbm}

\usepackage{graphicx,tikz, caption}

\theoremstyle{theorem}
\newtheorem{theorem}{Theorem}
\newtheorem{proposition}[theorem]{Proposition}

\newtheorem{lemma}[theorem]{Lemma}
\newtheorem*{question}{Question}

\theoremstyle{definition}

\theoremstyle{remark}
\newtheorem{remark}[theorem]{Remark}

\usepackage[draft]{fixme}
\fxsetup{theme=color}
\fxsetup{theme=color, mode=multiuser, noinline}
\FXRegisterAuthor{ss}{SS}{\color{red}Stefan}
\FXRegisterAuthor{rt}{RT}{\color{blue}Rekha}

\begin{document}

\title[]{Random Walks, Equidistribution \\and Graphical Designs}

\author[]{Stefan Steinerberger}
\address{Department of Mathematics, University of Washington, Seattle, WA 98195, USA} \email{steinerb@uw.edu}

\author[]{Rekha R. Thomas}
\address{Department of Mathematics, University of Washington, Seattle, WA 98195, USA} 
\email{rrthomas@uw.edu}

\thanks{S.S. is supported by the NSF (DMS-2123224) and the Alfred P. Sloan Foundation. R.T. is supported by the Walker Family Endowed Professorship at the University of Washington.}

\begin{abstract} Let $G=(V,E)$ be a $d$-regular graph on $n$ vertices and let $\mu_0$ be a probability measure on $V$. The act
of moving to a randomly chosen neighbor leads to a sequence of probability measures supported on 
$V$ given by $\mu_{k+1} = A D^{-1} \mu_k$
, where $A$ is the adjacency matrix and $D$ is the diagonal matrix of vertex 
degrees of $G$. 
Ordering the eigenvalues of $ A D^{-1}$ as $1 = \lambda_1 \geq |\lambda_2| \geq \dots \geq |\lambda_n| \geq 0$, it is well-known that the graphs for which $|\lambda_2|$ is small are those in which 
the random walk process converges quickly to the uniform distribution: for all initial probability measures $\mu_0$ and all $k \geq 0$,
$$ \sum_{v \in V} \left| \mu_k(v) - \frac{1}{n} \right|^2 \leq  \lambda_2^{2k}.$$
One could wonder whether this rate can be improved for specific initial probability measures $\mu_0$.
We show that if $G$ is regular, then for any $1 \leq \ell \leq n$, there exists a probability measure $\mu_0$ supported on at most $\ell$ vertices so that
$$ \sum_{v \in V} \left| \mu_k(v) - \frac{1}{n} \right|^2 \leq  \lambda_{\ell+1}^{2k}.$$
The result has applications in the graph sampling problem: we show that these measures have good sampling properties for reconstructing global averages.
\end{abstract}

\maketitle

\section{Introduction and Statement}
\subsection{Introduction.}  The behavior of random walks on a graph, and the speed with which it equidistributes, is a  well-studied problem in mathematics and computer science. Let $G=(V,E)$ be a finite undirected
graph with vertex set $V=\{v_1, \ldots, v_n\}$ and edge set $E$. We will assume throughout that $G$ is $d$-regular, meaning that exactly $d$ edges are incident at every vertex of $G$. Let $A \in \{0,1\}^{n \times n}$ denote the {\em adjacency matrix} of $G$ defined as $a_{ij} = 1$ if $\{i,j\} \in E$ and $0$ otherwise, and let $D$ 
be the diagonal matrix with $d_{ii} = d$ for all $i=1,\ldots,n$.
Starting with an initial probability measure $\mu_0$ supported on the vertices in $V$, the act of transporting probability mass to a randomly chosen neighbor leads to a sequence of probability  
measures $\{\mu_k\}$ supported on $V$ given by 
\begin{align}
 \mu_{k+1} = AD^{-1} \mu_k.
\end{align}

Since $AD^{-1} = (1/d) \cdot A$ is symmetric and 
doubly stochastic, its largest eigenvalue is $1$ and all other eigenvalues lie in $[-1,1]$. We may assume that the eigenvalues are labeled according to absolute value as 
\begin{align} \label{eq:eigenvalues}
    1 = \lambda_1 \geq |\lambda_2| \geq \dots \geq |\lambda_n| \geq 0.
\end{align}
$AD^{-1}$ has an orthonormal eigenbasis $\{\phi_1, \ldots, \phi_n\} \subset \mathbb{R}^n$,
with $\phi_i$ an eigenvector with eigenvalue $\lambda_i$. Note that $\phi_1 = \frac{\mathbbm{1}}{\sqrt{n}}$ where 
$\mathbbm{1}$ is the vector of all ones in $\mathbb{R}^n$. 
This information already allows for a quick and simple spectral analysis of random walks on $G$. 
Let $\frac{\mathbbm{1}}{{n}}$ denote the {\em uniform distribution} on $V$. Then applying the spectral theorem, one deduces that
\begin{align} \label{eq:convergence to uniform}
 \sum_{v \in V} \left| \mu_k(v) - \frac{1}{n} \right|^2 \leq  \lambda_2^{2k}.
 \end{align}
In other words, $\mu_k$ converges to the uniform distribution at a rate controlled by $|\lambda_2|$.
The argument also shows that this is, for `typical' measures $\mu_0$, the correct asymptotic rate as $k \rightarrow \infty$: the asymptotic rate is sharp up to (multiplicative) 
constants as soon as
$\left\langle \mu_0, \phi_2 \right\rangle \neq 0$.
This is one of the many ways in which $|\lambda_2|$ is a fundamental characteristic of the graph. It is also clear from the underlying picture that having a fast equidistribution rate is easier when there are many edges since that simplifies exploring the graph via a random walk. This naturally leads to the study of {\em expander graphs}, graphs with relatively few edges with excellent expansion properties. We refer to \cite{hoory-linial-wigderson} for a survey on expander graphs.

\subsection{Exceptional measures.}  One could naturally wonder about the existence of initial probability measures $\mu_0$ for which the rate of convergence of $\{\mu_k\}$ to the uniform distribution is faster than that in~\eqref{eq:convergence to uniform}. It will be evident from the arguments in the proof of our main theorem that such measures, if they do exist at all, have to be exceedingly rare.
If $\mu_0$ is the uniform measure $\frac{\mathbbm{1}}{n}$, then convergence is instantaneous since 
the uniform measure is invariant under the random walk: 
$$ AD^{-1} \frac{\mathbbm{1}}{n}= \frac{1}{d}A \frac{\mathbbm{1}}{n} = \frac{\mathbbm{1}}{n}.$$ 
However, $\frac{\mathbbm{1}}{n}$ is supported on all of $V$, which prompts the following.

\begin{question}
If the initial measure $\mu_0$ is restricted to be supported on at most $\ell$ vertices, what kind of convergence rate is possible?
\end{question}

If $\mu_0$ is supported on a single vertex, then one cannot in general hope for any rate faster than ${\lambda_2}^{2k}$. On the other hand, as we just saw, if $\mu_0$ can have support on all of $V$, then we can start with the uniform measure and convergence is instantaneous. This paper was motivated 
by trying to understand the intermediate regime, and our main result is the following.

\begin{theorem} \label{thm:main theorem}
Let $G=(V,E)$ be a connected, regular graph on $n$ vertices and let $1 \leq \ell \leq n-1$. There  exists a probability measure $\mu_0$ supported on at most $\ell$ vertices in $V$ 
such that the sequence of measures $\mu_{k+1} = AD^{-1} \mu_k$ satisfies
$$  \sum_{v \in V} \left| \mu_k(v) - \frac{1}{n} \right|^2 \leq  \lambda_{\ell+1}^{2k}.$$
\end{theorem}

While a random walk started from a single vertex $v$ converges to the 
uniform distribution at an exponential rate of at most
$\lambda_2^{2k}$ (this is the exact rate whenever $\phi_2(v) \neq 0$), our theorem says that there is always a probability measure $\mu_0$ supported on at most two vertices that converges at least at a rate controlled by $|\lambda_3|$. This becomes interesting when $|\lambda_3|$ is much smaller than 
$|\lambda_2|$, we refer to Fig.~\ref{fig:first example} and Fig. \ref{fig:two examples} for illustrations of Theorem~\ref{thm:main theorem}. 
An interesting special case is that of bipartite graphs which satisfy $\lambda_1 = |\lambda_2|=1$. If $\mu_0$ is supported on a single vertex, then it is impossible for $\mu_k$ to equidistribute: the support of $\mu_k$ oscillates between the two vertex sets in the bipartition, with the parity of $k$ predicting the half on which the measure is supported. 
In this case, a natural question is whether it is possible to achieve asymptotic uniform distribution with an  initial probability measure supported on two vertices instead. Theorem~\ref{thm:main theorem} guarantees that whenever $|\lambda_3| < 1$, there exists a probability measure $\mu_0$ supported in two vertices of a 
bipartite graph for which the arising random walk equidistributes (at a rate given by $|\lambda_3|$).

\begin{center}
    \begin{figure}[h!]
       \begin{tikzpicture}
       \node at (0,0) {\includegraphics[width=0.4\textwidth]{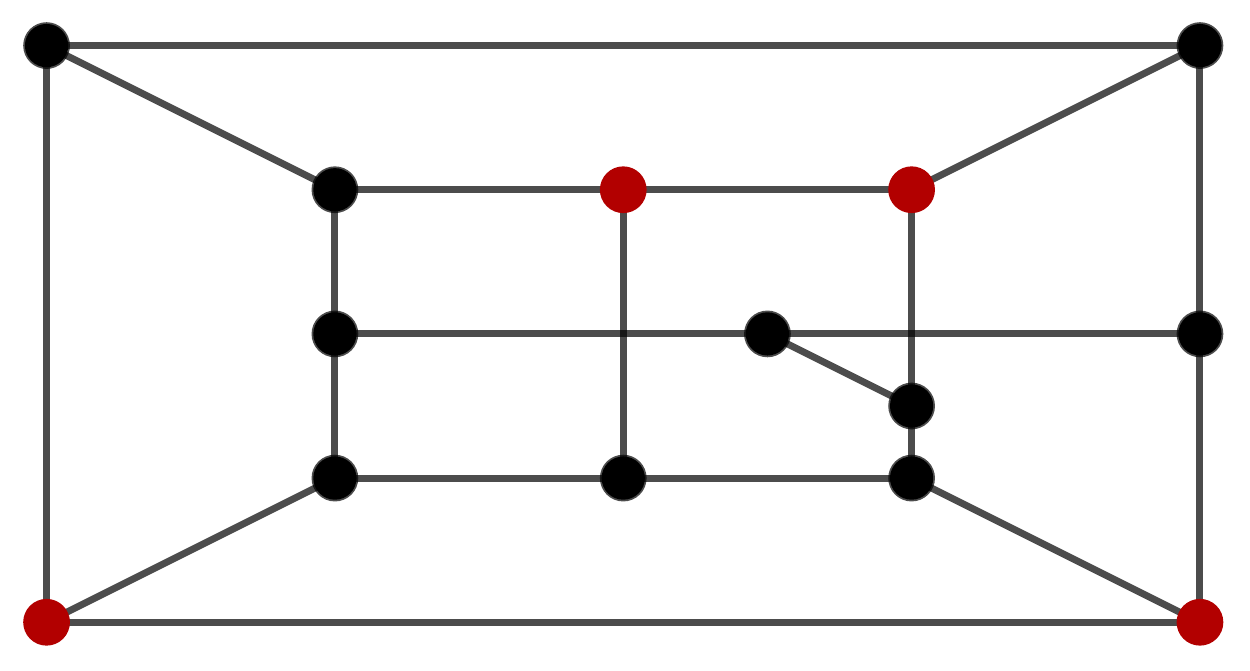}};
       \end{tikzpicture}
        \caption{Most probability measures on this graph converge to the uniform distribution at a rate given by $|\lambda_2| \sim 0.78$. By Theorem~\ref{thm:main theorem} there exists 
         $\mu_0$ supported on at most four vertices converging 
        at a faster rate given by $|\lambda_5| \sim 0.66$. Indeed, the uniform measure on the red vertices converges at a rate given by $|\lambda_9| \sim 0.47$.}
        \label{fig:first example}
    \end{figure}
\end{center}
\vspace{-10pt}

\subsection{Related results.} 
The exceptional measures in Theorem~\ref{thm:main theorem} were first introduced in a somewhat different setting under the name `Graphical Design' by Steinerberger \cite{steinerberger} as a structural analogue to {\em spherical designs}: spherical designs are subsets of points on the sphere with excellent sampling properties. It is a celebrated result of Delsarte, Goethals \& Seidel \cite{delsarte} that the cardinality of a set of points on the 
$d$-sphere $\mathbb{S}^d$ that is capable of exactly integrating all polynomials of degree at most $t$ needs to be comparable to the dimension of the vector space of polynomials of degree at most $t$. 
This has now been shown to be the correct order of magnitude by Bondarenko, Radchenko \& Viazovska \cite{bond}. Graphical designs average sufficiently smooth graph functions 
similarly to how spherical designs average polynomials of sufficiently small degree. One of the
main insights of \cite{steinerberger} is that in the graph setting, highly symmetric graphs admit 
unusually efficient graphical designs whose size is much smaller than the linear algebra prediction (similar 
to how the platonic bodies in $\mathbb{R}^3$ lead to unusually efficient spherical designs on $\mathbb{S}^2$).  
It is impossible to improve on the linear algebra setting in the continuous setting, see \cite{delsarte, steinerberger2}. Graphical designs were further explored by Babecki \cite{babecki} and Golubev \cite{golubev}.  The first general Existence Theorem was recently proven by Babecki \& Thomas \cite{thomas} who established a bijection 
between positively weighted graphical designs and the faces of a generalized {\em eigenpolytope}
of the graph. \\

There is also a philosophical relationship between our result and a result of Tchakaloff \cite{tch1, tch2, tch3} which has led to much subsequent work, see for example Bayer \& Teichmann \cite{bayer}, Putinar \cite{putinar} and Tur\'an \cite{tur}. It states that under suitable assumptions on the domain and the measure, integration of an $n-$dimensional space of functions can be achieved exactly with a positive weight supported on at most $n$ points. We will give a self-contained proof of such a result in the graph case (Lemma 1) using an argument reminiscent of Piazzon, Sommariva \& Vianello \cite{piazz}.

\vspace{-10pt}
\begin{center}
    \begin{figure}[h!]
       \begin{tikzpicture}
       \node at (0,0) {\includegraphics[width=0.25\textwidth]{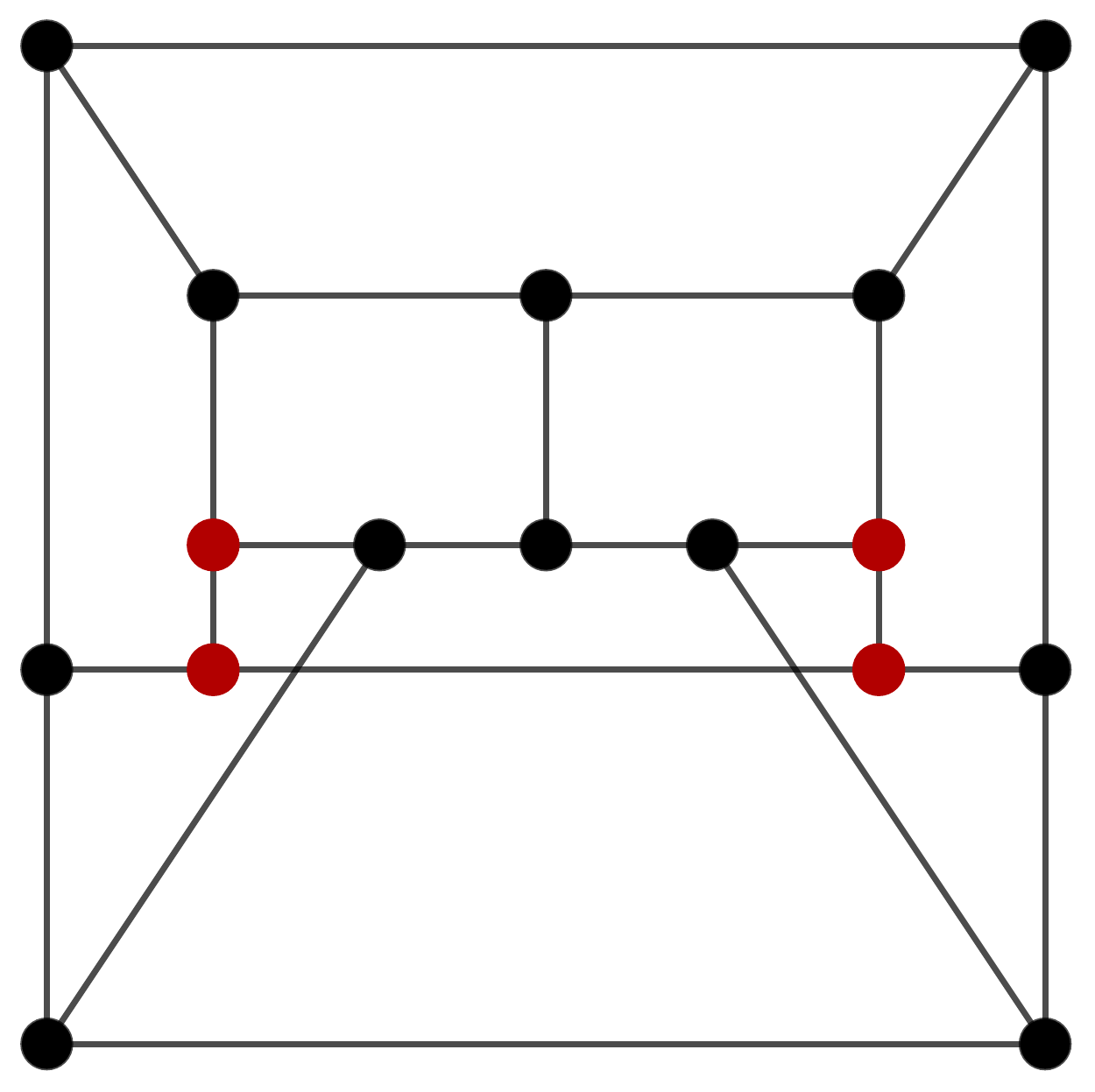}};
       \node at (5,0) {\includegraphics[width=0.25\textwidth]{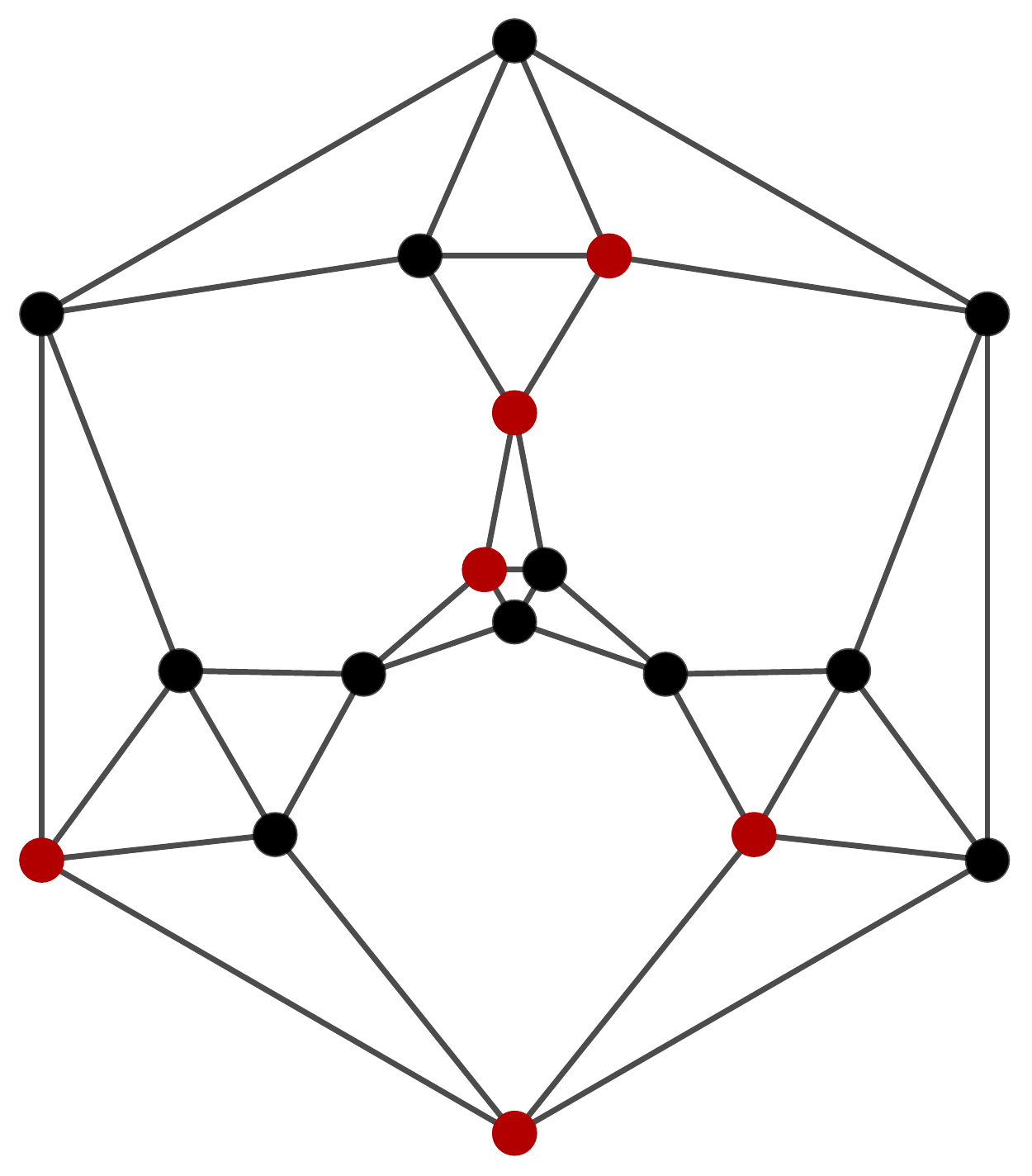}};
       \end{tikzpicture}
        \caption{Left: $|\lambda_2| \sim 0.83$. There exists a measure supported on at most 4 vertices decaying at a  rate given by $|\lambda_5| \sim 0.70$ (the uniform measure on the red vertices is an example). Right:  $|\lambda_2| = 0.75$. There exists a measure supported on at most 6 vertices decaying at a rate given by $|\lambda_7| = 0.5$. Indeed, the uniform measure on the red vertices decays at rate determined by  $|\lambda_{11}| \sim 0.25$.}
        \label{fig:two examples}
    \end{figure}
\end{center}

\vspace{-10pt}

This paper is organized as follows. In Section~\ref{sec:proof} we prove Lemma~\ref{lem:caratheodory} that shows the existence of initial measures $\mu_0$ with the properties needed to prove Theorem~\ref{thm:main theorem}. Using it we 
prove our main theorem. Lemma 2 can also be used to derive 
results for graph sampling as we show in Section~\ref{sec:sampling}.

\section{Proof of the Theorem} \label{sec:proof}
\subsection{Constructing Initial Measures.}
Recall that $G$ is a $d$-regular graph with $n$ vertices, the eigenvalues of $AD^{-1}$ lie in $[-1,1]$ and 
are ordered as in \eqref{eq:eigenvalues} 
with a corresponding set of orthonormal eigenvectors $\{\phi_1, \ldots, \phi_n\}$ that form a basis 
of $\mathbb{R}^n$. 
The eigenvector of $\lambda_1$ is often called the {\em trivial} eigenvector and is  
$$\phi_1 = \left(\frac{1}{\sqrt{n}}, \frac{1}{\sqrt{n}},  \dots, \frac{1}{\sqrt{n}}\right) = \frac{\mathbbm{1}}{\sqrt{n}}.$$

\begin{lemma} \label{lem:caratheodory}
For every $2 \leq \ell \leq n-1$ there exists a vector $0 \neq w \in \mathbb{R}_{\geq 0}^n$ such that
\begin{enumerate}
\item $w$ has at most $\ell$ nonzero entries, and 
\item for all $2 \leq j \leq \ell$, we have
$ \left\langle \phi_j, w \right\rangle = 0.$
\end{enumerate}
\end{lemma}

\begin{proof}
Fix $\ell$ such that $2 \leq \ell \leq n-1$ and let $M$ be the 
$\ell \times n$ matrix (of rank $\ell$) whose rows are $\phi_1^\top, \ldots, \phi_\ell^\top$. 
Let $e_1$ be the first standard basis vector in $\mathbb{R}^\ell$. 
The statement of the lemma is equivalent to saying that 
the system 
\begin{align} \label{eq:nnwt design}
     Mw =e_1, \,\,\, w \geq 0
\end{align}
has a solution supported on at most $\ell$ coordinates.
Indeed, if there exists $0 \neq w \in \mathbb{R}^n_{\geq 0}$ 
such that $\langle \phi_j, w\rangle = 0$ for $j=2, \ldots, \ell$ then 
$\langle \phi_1, w \rangle > 0$ and we may scale $w$ to obtain a solution of \eqref{eq:nnwt design}. 
Conversely, any solution to \eqref{eq:nnwt design} provides a nonzero $w \in \mathbb{R}^n_{\geq 0}$ orthogonal to 
$\phi_2, \ldots, \phi_\ell$.
Let $K$ denote the polyhderal cone spanned by the columns of $M$, i.e., 
$$K = \{ M x \,:\, x \geq 0\} \subset \mathbb{R}^{\ell}.$$ 
The dimension of $K$ is $\ell$ since $\textup{rank}(M) = \ell$. 
The system \eqref{eq:nnwt design} has a solution 
if and only if $e_1 \in K$.
If $e_1 \in K$ then by Caratheodory's theorem \cite{Caratheodory}, 
$e_1$ lies in a subcone spanned by at most $\dim(K) = \ell$ columns of $M$ and so 
there is a solution $w$ to \eqref{eq:nnwt design} 
with at most $\ell$ non-zero entries and we are done.
If $e_1 \not \in K$, i.e., \eqref{eq:nnwt design} does not have a solution, then 
by the Farkas Lemma \cite{farkas} 
there exists $y \in \mathbb{R}^\ell$ such that 
\begin{align} \label{eq:farkas}
    y^\top M \geq 0, \textup{ and } y^\top e_1 = y_1 < 0.
\end{align}
Since positively scaling $y$ does not affect \eqref{eq:farkas}, we may assume that $y_1 = -1$, 
and so $y^\top M \geq 0$ says that 
\begin{align} 
-\phi_1 + y_2 \phi_2 +  y_3 \phi_3 + \ldots +  y_{\ell} \phi_\ell  \geq 0.
\end{align}
Taking the dot product on both sides with $\phi_1$, and using the fact that the $\{\phi_i\}$ 
are orthonormal, we get $-1 \geq 0$, which is a contradiction.
Therefore, $e_1 \in K$.
\end{proof}

\begin{remark}
Note that there is always a solution to $M w = e_1$ supported on at most $\ell$ coordinates since $M$ has full 
row rank. More precisely, there is a $\ell \times \ell$ submatrix $N$ of $M$ such that up to permutation of coordinates, $w = (N^{-1}e_1, 0)$ is a solution to \eqref{eq:nnwt design} with 
at most $\ell$ nonzero entries. The content of Lemma~\ref{lem:caratheodory} is that there is a 
nonnegative $w$ satisfying the required conditions. 
\end{remark}

\begin{remark} \label{rem:ordering doesn't matter}
The ordering of eigenvalues as in \eqref{eq:eigenvalues} 
was not relevant in the proof of Lemma~\ref{lem:caratheodory}. The statement of 
the lemma holds for any 
ordering of the eigenvectors (and hence eigenvalues) 
as long as $\phi_1 = \frac{\mathbbm{1}}{\sqrt{n}}$ and $\lambda_1 = 1$.
\end{remark}

\begin{remark} \label{rem:lemma for L}
We also observe that Lemma 2 did not rely on any special properties of the vectors $\phi_j$ except
orthogonality. As such, it could also be applied to eigenvectors of the Laplacian of $G$ given by $L = D-A$ 
or the eigenvectors of other appropriate symmetric matrices.
\end{remark}

\begin{remark}
While the inequality  $ \# \left\{i: w_i \neq 0 \right\} \leq \ell$ is needed in general (see \cite{thomas} for examples where this bound is tight), there are examples where $ \# \left\{i: w_i \neq 0 \right\}$ is much smaller than $\ell$ while still satisfying 
$ \left\langle \phi_j, w \right\rangle = 0$ for all $2 \leq j \leq \ell$.
\end{remark}

\subsection{Proof of Theorem~\ref{thm:main theorem}.}

\begin{proof} Let $w$ be a vector satisfying the conditions of Lemma~\ref{lem:caratheodory}. 
We can assume without loss of generality that $\|w\|_{1} = 1$ which means that
$\mu_0 := w$ is a probability measure supported on at most $\ell$ vertices of $G$. Since 
$\{ \phi_1, \ldots, \phi_n\}$ is an orthonormal basis of $\mathbb{R}^n$,  
$\langle w, \phi_j\rangle =0$ for $2 \leq j \leq \ell$, $\|w\|_{1} = 1$ and 
$\phi_1 = \frac{\mathbbm{1}}{\sqrt{n}}$, we have
\begin{align*}
w = \sum_{j=1}^{n} \left\langle w, \phi_j \right\rangle \phi_j  = \langle w, \phi_1\rangle \phi_1 + 
\sum_{j=\ell+1}^{n} \left\langle w, \phi_j \right\rangle \phi_j 
= \frac{\mathbbm{1}}{n} +  \sum_{j=\ell+1}^{n} \left\langle w, \phi_j \right\rangle \phi_j.
\end{align*}
The symmetric matrix $AD^{-1}$ has the diagonalization $A D^{-1} = U \Lambda U^\top$ where 
$U$ has the orthonormal columns $\phi_1, \ldots, \phi_n$ and $\Lambda = \textup{diag}(\lambda_1, 
\ldots, \lambda_n)$. Therefore, 
$$ \mu_k = (A D^{-1})^k w =  U\Lambda^k U^\top w = \frac{\mathbbm{1}}{n} +  \sum_{j=\ell + 1}^{n} \lambda_j^k \langle w, \phi_j \rangle \phi_j.$$
Using the orthonormality of the eigenvectors and the ordering \eqref{eq:eigenvalues} we get that 
\begin{align*}
 \sum_{v \in V} \left( \mu_k(v) - \frac{1}{n} \right)^2 &= \left\| \mu_k - \frac{\mathbbm{1}}{n} \right\|_{2}^2 
  =\left\|    \sum_{j=\ell + 1}^{n} \lambda_j^k \left\langle w, \phi_j \right\rangle \phi_j \right\|_{2}^2 \\
  & =  \sum_{j=\ell + 1}^{n} {\lambda_j}^{2k} \left\langle w, \phi_j \right\rangle^2 
   \leq  {\lambda_{\ell+1}}^{2k} \sum_{j=\ell+1}^{n}  \left\langle w, \phi_j \right\rangle^2 \\
   &\leq  {\lambda_{\ell+1}}^{2k} \sum_{j=1}^{n}  \left\langle w, \phi_j \right\rangle^2
 =  {\lambda_{\ell+1}}^{2k} \left\| w \right\|_{2}^2 \leq {\lambda_{\ell+1}}^{2k}.
 \end{align*}
The last inequality follows from the fact that $w$ is a probability measure, and hence,
$$\| w\|_{2} =\left( \sum_{v \in V} w(v)^2 \right)^{1/2} \leq \left( \sum_{v \in V} w(v) \right)^{1/2} = \|w\|_1 = 1.$$
\end{proof}

\begin{remark}
We note that the bound $\left\| w \right\|_{2}^2 \leq 1$, can be suboptimal. In particular, if $w$ is evenly distributed over $\ell$ vertices, then this would lead to an improvement by a factor of up to $\left\| w \right\|_{2}^2 = \ell^{-1}$. 
\end{remark}

\begin{remark} \label{rem:more general conclusion}
If we had not used the fact that 
 $\langle w, \phi_j \rangle = 0$ for $2 \leq j \leq \ell$ up front in the proof, then we would have 
gotten the identity
$$  \sum_{v \in V} \left( \mu_k(v) - \frac{1}{n} \right)^2 = \left \| \sum_{j=2}^{n} {\lambda_j}^{2k} \left\langle w, \phi_j \right\rangle^2 \right \|^2_2.$$
Hence, the asymptotic behavior of the left-hand side as $k \rightarrow \infty$ is 
controlled by $\lambda_a^{2k}$, where
$ a = \min \left\{ j > 1: \left\langle \mu_0, \phi_j\right\rangle \neq 0 \right\}.$ 
 The orthogonality requirements on $w$ with respect to the eigenbasis 
$\{\phi_i\}$ show that measures $w$ that give faster convergence rates than 
${\lambda_2}^{2k}$ are rare in the sense that such measures form a lower-dimensional set in the probability simplex.
\end{remark}

\begin{remark}
If $\ell = 1$, then we can pick $\mu_0$ to be the Dirac measure on a vertex of $G$ and the statement of Theorem~\ref{thm:main theorem} is just the well-known expression \eqref{eq:convergence to uniform}. 
If $\ell = 2$, then we require the initial probability measure 
$\mu_0$ to be orthogonal to $\phi_2$. Since $\langle \phi_1, \phi_2 \rangle = 0$ and $\phi_1 = \frac{\mathbbm{1}}{\sqrt{n}}$ is positive, it must be that $\phi_2$ has two non-zero coordinates of opposite signs. Suppose the $i$th coordinate of $\phi_2$ is 
$\alpha > 0$ and the $j$th coordinate is $\beta < 0$.  Then we may pick $w \in \mathbb{R}^{n}_{\geq 0}$ as $w_i = -\beta, w_j = \alpha$, $w_k = 0$ for all $k \neq i,j$ and 
then normalize to obtain an initial measure $\mu_0$ supported on two vertices. By 
Theorem~\ref{thm:main theorem}, a random walk initialized with $\mu_0$ will converge to the uniform distribution at rate controlled by $|\lambda_3|$. 
A direct argument as above does not seem possible for higher values of $\ell$ since the sign patterns of the eigenvectors start to interact in somewhat nontrivial ways. However, an initial measure can always be obtained by solving the linear program 
$$\textup{minimize} \{ \langle c, w \rangle \,:\, Mw = e_1, w \geq 0 \}$$ 
coming from the constraints in \eqref{eq:nnwt design}, for a suitable vector $c$. The optimum will occur at a vertex of the polytope $\{w \in \mathbb{R}^n\,:\, Mw = e_1, w \geq 0\}$ whose 
support will have size at most $\ell = \textup{rank}(M)$. 
\end{remark}

\section{Application to Graph Sampling} 
\label{sec:sampling}
Graph Sampling refers to the problem of trying to estimate, for a given function
$f:V \rightarrow \mathbb{R}$, its global average on $G$ by sampling $f$ 
in a few vertices and using a weighted average of the function values on these vertices. 
Mathematically, we want a subset $U \subset V$ and weights $a_u \in \mathbb{R}, u \in U$ such that 
\begin{align} \label{eq:sample}
    \frac{1}{|V|} \sum_{v \in V} f(v) \sim \sum_{u \in U} a_u f(u). 
\end{align} 
The main problems are (i) how should one choose $U$, (ii) 
how should one choose the weights $a_u$,  and (iii) what sort of approximation
guarantees can one obtain? 
 Lemma~\ref{lem:caratheodory} implies that good sample points with 
non-negative weights 
always exist.  

\begin{proposition} \label{prop:graph sampling}
Let $G=(V,E)$ be a connected, regular graph on $n$ vertices and let $1 \leq \ell \leq n-1$. There exists a subset $U \subset V$ with $\# U \leq \ell$ and nonnegative weights $a_u \geq 0$ such that for all functions $f:V \rightarrow \mathbb{R}$
$$ \left| \frac{1}{|V|} \sum_{v \in V} f(v) -  \sum_{u \in U} a_u f(u) \right| \leq  \left(\sum_{i = \ell+1}^n \left\langle \phi_i, f \right\rangle^2\right)^{1/2}.$$
\end{proposition}

\begin{proof}
Let $w$ denote the initial measure guaranteed by Lemma~\ref{lem:caratheodory},  
$U$ denote its support and $a_u = w(u)$ for all $u \in U$. 
Set the mean value of $f$ to be 
\begin{align} \label{eq:mean value}
\bar{f} : = \frac{\mathbbm{1}^\top f}{n} = \frac{1}{|V|} \sum_{v \in V} f(v).
\end{align}
Then, using the fact that $\phi_1 = \frac{\mathbbm{1}}{\sqrt{n}}$ 
we see that 
\begin{align} \label{eq:f}
f = \sum_{i=1}^{n} \left\langle \phi_i, f \right\rangle \phi_i = 
\bar{f} \mathbbm{1} + \sum_{i=2}^n \left\langle \phi_i, f \right\rangle \phi_i.
\end{align}
From \eqref{eq:f} and the facts that $\mathbbm{1}^\top w = 1$ 
and $\left\langle \phi_j, w \right\rangle = 0 \quad \forall~2 \leq j \leq \ell$, we get  that 
\begin{align}\label{eq:weighted sum}
 \sum_{u \in U} a_u f(u) = \left\langle f, w \right\rangle 
 =  \left\langle  \overline{f} \mathbbm{1} + 
 \sum_{i=2}^{n} \left\langle \phi_i, f \right\rangle \phi_i, w \right\rangle
% &= \overline{f} \mathbbm{1}^\top w + 
% \sum_{i=\ell + 1}^{n} \left\langle  \left\langle \phi_i, f \right\rangle \phi_i, w \right\rangle 
 =  \overline{f} + \sum_{i=\ell+1}^{n}  \left\langle \phi_i, f \right\rangle \left\langle \phi_i, w \right\rangle.
 \end{align}
 Combining \eqref{eq:mean value} and \eqref{eq:weighted sum}
we conclude that 
 \begin{align}
     \frac{1}{|V|} \sum_{v \in V} f(v) -  \sum_{u \in U} a_u f(u) = 
     - \sum_{i=\ell+1}^{n} \left\langle  \phi_i, f \right\rangle \left\langle \phi_i, w \right\rangle.
 \end{align}
Then, using the Cauchy-Schwarz inequality we get the desired result: 
\begin{align*}
    \left| \frac{1}{|V|} \sum_{v \in V} f(v) - \sum_{u \in U} a_u f(u) \right| &= 
    \left| \sum_{i=\ell+1}^{n} \left\langle  \phi_i, f \right\rangle \left\langle \phi_i, w \right\rangle  \right| \\
    &\leq  \left(\sum_{i = \ell+1}^n \left\langle \phi_i, f \right\rangle^2\right)^{1/2}  \left(\sum_{i = \ell+1}^n \left\langle \phi_i, w \right\rangle^2\right)^{1/2} \\
    &\leq  \left(\sum_{i = \ell+1}^n \left\langle \phi_i, f \right\rangle^2\right)^{1/2}  \|w\|_{2}
     \leq \left(\sum_{i = \ell+1}^n \left\langle \phi_i, f \right\rangle^2\right)^{1/2}.
\end{align*}
\end{proof}

\begin{remark}
The expression on the right-hand side of the inequality in Proposition~\ref{prop:graph sampling} is the square of the length of the projection of $f$ on the subspace spanned by 
$\phi_{\ell + 1}, \ldots, \phi_n$. This length is a measure of the smoothness of $f$ 
in the sense of whether or not $f$ is mainly comprised of low-frequency eigenvectors, namely those $\phi_i$ with $i$ small. If yes, then the approximation provided by sampling at $U$ is very good. If $f$ is highly oscillatory, then the right-hand
side will not be much smaller than $\|f\|_{2}$, and in that 
case, one cannot hope for an accurate reconstruction of the average $\bar{f}$ by sampling at $U$. In fact, if we obtain a large sampling error despite using an 
intial measure provided by Theorem~\ref{thm:main theorem}, then it provides evidence that 
$f$ is highly oscillatory and has a large amount of its $\ell^2-$energy on high-frequency eigenvectors of $AD^{-1}$. 
\end{remark}

\begin{remark} Proposition~\ref{prop:graph sampling} allows one to custom-tailor 
$U$ in the presence of prior knowledge of the function $f$. For example, if 
one knew a priori that $f$ was mainly comprised of $\ell-1$ mid-frequency eigenvectors, then we could change the ordering of eigenvectors to make these appear after $\phi_1$. Then Lemma~\ref{lem:caratheodory} would provide an initial measure supported on at most $\ell$ vertices that is orthogonal to all the chosen mid-frequency eigenvectors 
and hence a $U$ for which the right-hand side in Proposition~\ref{prop:graph sampling} is small. 
\end{remark}

\begin{remark} As already noted in Remark~\ref{rem:lemma for L},  
 Lemma~\ref{lem:caratheodory} holds for the Graph Laplacian 
 $L = D-A$ whose top eigenvalue is $1$ with $\phi_1 = \frac{\mathbbm{1}}{\sqrt{n}}$. Further, $L$ has an orthonormal eigenbasis. This implies that 
 Proposition~\ref{prop:graph sampling}, with the same proof, 
 also holds for non-regular graphs provided we replace the eigenvectors
 $\phi_i$ by eigenvectors of $D-A$.
 \end{remark}

\begin{remark}
Proposition~\ref{prop:graph sampling} should be compared to the results
of Linderman \& Steinerberger \cite{george} who connected integration error with the speed of equidistribution of
the induced random walk. In comparison, the argument for Proposition~\ref{prop:graph sampling} is simple and self-contained. We also refer to \cite{pes1, pes2, t} for  related results. 
 \end{remark}

\end{document}